\documentclass[11pt]{amsart}
\usepackage{hyperref}

 \usepackage{bm}
  \newcommand{\n}{{\bf n}}

\newcommand{\cX}{\mathcal{X}}

\newcommand{\G}{\mathcal{G}}
\newcommand{\X}{\mathcal{X}}
\newcommand{\cI}{\mathcal{I}}
\newcommand{\CI}{\mathcal{I}}
\newcommand{\CZ}{\mathcal{Z}}

\newcommand{\Y}{\mathcal{Y}}
\newcommand{\cZ}{\mathcal{Z}}
\newcommand{\gG}{\Gamma}
\newcommand{\C}{\mathbb{C}}

\newcommand{\R}{\mathbb{R}}
\newcommand{\E}{\mathbb{E}}
\newcommand{\N}{\mathbb{N}}

\newcommand{\Z}{\mathbb{Z}}

\renewcommand{\n}{{\bf n}}
\newcommand{\uepsilon}{\bm{\epsilon}}

\DeclareMathOperator{\vdc}{-vdC}

\newcommand{\norm}[1]{\left\Vert #1\right\Vert}
\newcommand{\nnorm}[1]{\lvert\!|\!| #1|\!|\!\rvert}
\theoremstyle{plain}
\newtheorem{theorem}{Theorem}[section]
\newtheorem{lemma}[theorem]{Lemma}
\newtheorem{proposition}[theorem]{Proposition}

\newtheorem*{theoremA'}{Theorem A'}
\newtheorem*{theoremB'}{Theorem B'}
\newtheorem*{theoremC'}{Theorem C'}

\newtheorem{problem}{Problem}
\newtheorem*{theorem*}{Theorem}

\newtheorem{corollary}[theorem]{Corollary}

\theoremstyle{definition}

\theoremstyle{remark}

\begin{document}
\title[\tiny{Pointwise convergence for cubic and polynomial multiple ergodic averages }]{Pointwise convergence for cubic and polynomial multiple ergodic averages  of non-commuting transformations}
\author{Qing Chu}
\address[Qing Chu]{Universit\'e Paris-Est Marne-la-Vall\'ee, Laboratoire d'analyse et de math\'ematiques appliqu\'ees, UMR
CNRS 8050, 5 Bd Descartes, 77454 Marne la Vall\'ee Cedex, France} \email{qing.chu@univ-mlv.fr}

\author{Nikos Frantzikinakis}
\address[Nikos Frantzikinakis]{University of Crete, Department of mathematics, Knossos Avenue, Heraklion 71409, Greece} \email{frantzikinakis@gmail.com}

\begin{abstract}
We  study the limiting behavior of multiple ergodic averages involving several not necessarily commuting measure preserving transformations. We work on two types of averages, one that uses iterates
along combinatorial
parallelepipeds, and another that uses iterates along  shifted polynomials. We prove pointwise convergence in
 both cases, thus answering  a question of
 I.~Assani in the former case, and  extending results of B.~Host-B.~Kra and A.~Leibman in the latter case.
 Our argument is based on  some elementary uniformity estimates of general bounded sequences,   decomposition results in ergodic theory, and equidistribution results on nilmanifolds.
\end{abstract}


\thanks{The  second author was partially supported by
 Marie Curie IRG  248008.}

\subjclass[2000]{Primary: 37A45; Secondary: 28D05, 11B30. }


\maketitle

\section{Main results}
In this paper we study the limiting behavior, in the mean and pointwise,  of  multiple ergodic averages involving
 measure preserving transformations that do not necessarily commute. We focus our attention on two such types, special cases of which
 have previously attracted some attention.
 One involves  iterates taken along combinatorial parallelepipeds, and the other involves iterates taken along shifted polynomials.
\subsection{Cubic Averages}
For $k\in\N$ we set
 $$
 V_k\mathrel{\mathop:}=\{0,1\}^k \quad \text{and} \quad
 V^*_k\mathrel{\mathop:}=V_k \setminus \{ {\bf 0}\}
 $$
where  ${\bf 0}\mathrel{\mathop:}=(0,0,\cdots, 0)$.
Let $(X,\X,\mu)$ be a probability space\footnote{Throughout the text all probability spaces are assumed to be Lebesgue.},
and for $\uepsilon\in V^*_k$ let
   $T_{\uepsilon}\colon X\to X$ be  measure preserving transformations and  $f_{\uepsilon}\in L^\infty(\mu)$ be functions.
We are going to study the limiting behavior of certain multiple ergodic averages taken along $k$-dimensional combinatorial parallelepipeds of iterates of the transformations $T_{\uepsilon}$.
More precisely, the \emph{cubic averages of dimension $k$} are given by
\begin{equation}\label{E:A_k}
A_{k,N}(T_{\uepsilon},f_{\uepsilon})(x)\mathrel{\mathop:}=\frac{1}{N^k}\sum_{\n\in [1,N]^k} \prod_{{\uepsilon}\in V^*_k} f_{\uepsilon}(
T_{\uepsilon}^{{\uepsilon}\cdot \n}x)
\end{equation}
where for $\uepsilon=(\epsilon_1,\ldots,\epsilon_k)\in V_k$ and  $\n=(n_1,\ldots,n_k)\in \N^k$   we define
$$
{\uepsilon} \cdot \n\mathrel{\mathop:}=\epsilon_1n_1+\cdots +\epsilon_kn_k.
$$
For instance, the cubic averages of dimension $1$ are the ergodic averages,
the  cubic averages of dimension $2$ are defined  by
 $$
 \frac{1}{N^2} \sum_{1\leq m,n\leq N} f_1(T_1^{m}x)\cdot f_2(T_2^{n}x)\cdot f_3(T_3^{m+n}x),
 $$
and  the  cubic averages of dimension $3$ are similarly defined, using iterates of $7$ transformations taken
along the combinatorial parallelepipeds $m,n,r,m+n,m+r,n+r,m+n+r$.

The averages $A_{k,N}(T_{\uepsilon},f_{\uepsilon})(x)$  are closely linked to
 the Gowers-Host-Kra seminorms
$\nnorm{\cdot }_k$ that have been used extensively in ergodic theory to find convenient majorants for various other multiple ergodic averages. In \cite{HKa} it is shown that for ergodic systems $(X,\X,\mu,T)$, and  real valued functions $f\in L^\infty(\mu)$, we have
$$
\nnorm{f}_{k}^{2^{k}}=\lim_{N\to\infty}\int f\cdot A_{k,N}(T,f)\ d\mu
$$
where $A_{k,N}(T,f)$ is defined by letting $T_{\uepsilon}=T$ and $f_{\uepsilon}=f$  in \eqref{E:A_k}
for every ${\uepsilon}\in V^*_k$. This identity also holds for non-ergodic systems once the seminorms $\nnorm{\cdot}_k$ are appropriately defined.

The study of the limiting behavior of the averages \eqref{E:A_k} was initiated  by V. Bergelson in \cite{Be00}, where convergence in $L^2(\mu)$ was shown in dimension $2$ under the extra assumption that all the transformations are equal. Under the same assumption, Bergelson's result
was   extended by B.~Host and B.~Kra for cubic averages of dimension $3$ in \cite{HK04}, and for
arbitrary dimension $k$ in \cite{HKa}.
More recently  in \cite{As},  I.~Assani
established pointwise convergence for cubic averages of  arbitrary dimension $k$ when all
the transformations are equal. In the same article, and prior to this in \cite{As03} and \cite{As09}, convergence  for general, not necessarily
 commuting transformations, was studied for the first time.
 Pointwise convergence   was established for $2$-dimensional averages, and
some partial results were obtained for dimensions greater than $2$, including convergence when all the transformations are weak mixing. In this article we complete this study by proving  pointwise convergence for the cubic averages  of  arbitrary dimension.
\begin{theorem}\label{T:2}
Let $k\in \N$, $(X,\X,\mu)$ be a probability space,  and for ${\uepsilon}\in V^*_k$ let $T_{\uepsilon}\colon X\to X$ be measure preserving transformations,
and $f_{\uepsilon} \in L^\infty(\mu)$ be functions.
Then  the cubic averages of dimension $k$, given by  \eqref{E:A_k}, converge pointwise as $N\to \infty$.
\end{theorem}
It is interesting to
 contrast the limiting behavior  of the cubic averages with some other similar looking averages.
 To begin with, the
 averages
  $ \frac{1}{(N-M)^2}\sum_{M<m,n\leq N} f_1(T_1^mx)\cdot f_2(T_2^nx)\cdot f_3(T_3^{m+n}x)$, and their higher dimensional relatives,
do not in general   converge pointwise (for an example when $f_2$=$f_3$=$1$ see \cite{JR79}). On the other hand, our argument
can be easily adapted to prove convergence in $L^2(\mu)$ for such averages.
As for  the   averages
$
\frac{1}{N^2}\sum_{1\leq m,n\leq N} f_1(T^nx)\cdot f_2(S^mx)\cdot  f_3(T^nS^mx),
$
and the ``diagonal averages"
$
\frac{1}{N}\sum_{n=1}^N f_1(T^nx)\cdot f_2(S^nx),
$
it is known that they do not converge in general, even in $L^2(\mu)$,   unless one makes some commutativity assumption about the transformations $T$ and $S$
(for counterexamples, see \cite{Lei02b} for the former, and
 \cite{Bere85} or   \cite{BL04}  for the latter).
 In fact, even under the assumption that all transformations commute, pointwise convergence
 of these averages
 and their higher dimensional relatives is not known.

A key concept that underlies the convergence result of Theorem~\ref{T:2}
is   the \emph{characteristic factors}, meaning a collection of
$T_{\uepsilon}$-invariant sub-$\sigma$-algebras  $\Y_{\uepsilon}$, having the property
 that the difference
$A_{k,N}(T_{\uepsilon},f_{\uepsilon})(x)-A_{k,N}(T_{\uepsilon},\tilde{f}_{\uepsilon})(x)$, where
$\tilde{f}_{\uepsilon}=\E(f_{\uepsilon}|\Y_{\uepsilon})$, converges pointwise  to $0$.
Our main goal is to make a suitable choice so that the corresponding factor systems have very special algebraic structure. This is done by controlling the averages $A_{k,N}(T_{\uepsilon},f_{\uepsilon})$ by certain seminorms (their precise definiton is given in Section~\ref{SS:Seminorms}), thus obtaining the following result:

\begin{theorem}\label{T:2'}
Let $k\in \N$, $(X,\X,\mu)$ be a probability space,  and for ${\uepsilon}\in V^*_k$ let $T_{\uepsilon}\colon X\to X$ be  measure preserving transformations,
and $f_{\uepsilon} \in L^\infty(\mu)$ be functions. Furthermore, suppose that  $\nnorm{f_{\uepsilon}}_{k,T_{\uepsilon}}=0$ for some ${\uepsilon}\in V^*_k$.
Then  the cubic averages of dimension $k$,  given by \eqref{E:A_k}, converge pointwise to $0$ as $N\to \infty$.
\end{theorem}
In fact we give explicit bounds relating the pointwise limiting behavior of
the cubic averages \eqref{E:A_k} and the seminorms  $\nnorm{f_{\uepsilon}}_{k,T_{\uepsilon}}$ (see Corollary~\ref{C:key'}).

Using different terminology, Theorem~\ref{T:2'} states that  the factors $\cZ_{k-1,T_{\uepsilon}}$,
defined in Section~\ref{SS:Z_k}, are characteristic factors  for pointwise convergence of the
averages \eqref{E:A_k}.

To prove  Theorem~\ref{T:2'} we simplify and extend  to our particular context an argument given by Assani in \cite{As}.
To prove Theorem~\ref{T:2} we combine Theorem~\ref{T:2'} with the decomposition result of Proposition~\ref{P:ApprNil}
 (which was proved in \cite{CFH09} using the structure theorem of \cite{HKa}). We eventually reduce matters
 to a known convergence property of nilsequences (all notions are defined in  Section~\ref{S:Background}).

\subsection{Polynomial averages}
We are going to  generalize some convergence results of B.~Host and B.~Kra \cite{HK05b} and A.~Leibman \cite{Lei05c} that involve multiple ergodic averages of a single transformation to the case that involves
several not necessarily commuting transformations.


\begin{theorem}\label{T:PoliesConv1}
Let $\ell\in \N$, and  $(X,\X,\mu)$ be a probability space. For   $i=1,\ldots,\ell$ let
 $T_i\colon X\to X$ be measure preserving transformations,  $f_i \in L^\infty(\mu)$ be functions,  $p_i\in \Z[t]$ be non-constant polynomials such  that $p_i-p_j$ is non-constant
 for $i\neq j$, and
  $b\colon \N\to \N$ be a sequence  such that $b(N)\to \infty$ and  $b(N)/N^{1/d}\to 0$ as $N\to\infty$,
 where $d$ is the maximum degree of the polynomials $p_i$.\footnote{The second condition  guarantees
 that the contribution  of several boundary terms is negligible. For instance, for every bounded sequence $(a(n))_{n\in\N}$ and polynomial $p\in \Z[t]$ with degree at most $d$,   the difference of the averages
 $\E_{1\leq n\leq b(N),p(n)\leq m\leq N+p(n)} a(n)$ and  $\E_{1\leq n\leq b(N), 1\leq m\leq N} a(n)$ goes to $0$ as $N\to\infty$.}
Then  the averages
\begin{equation}\label{E:Polynomial}
\frac{1}{Nb(N)}\sum_{1\leq m\leq N, 1\leq n\leq b(N)} f_1(T_1^{m+p_1(n)}x)\cdot \ldots\cdot f_\ell(T_\ell^{m+p_\ell(n)}x)
\end{equation}
converge pointwise as $N\to \infty$.
\end{theorem}
Using this  result for $\ell+1$ in place of $\ell$, letting  $T_0=\cdots=T_{\ell}=T$,  $p_{0}=0$, and   integrating with respect to $\mu$,   we deduce
that  the averages
 \begin{equation}\label{E:rtyu}
\frac{1}{N}\sum_{n=1}^N f_1(T^{p_1(n)}x)\cdot \ldots\cdot f_\ell(T^{p_\ell(n)}x)
\end{equation}
converge weakly in $L^2(\mu)$ as $N\to\infty$. This   recovers
 one of the main results from \cite{HK05b}. Let us remark though that  we  were not able to deduce from  Theorem~\ref{T:PoliesConv1}  anything useful regarding
 the well known  open problem of convergence  (weakly, in the mean, or pointwise)  of the averages
$\frac{1}{N}\sum_{n=1}^N f_1(T_1^{p_1(n)}x)\cdot \ldots\cdot f_\ell(T_\ell^{p_\ell(n)}x)$
for general commuting measure preserving transformations $T_1,\ldots,T_\ell$.

 A key ingredient in the proof of Theorem~\ref{T:PoliesConv1}  is  the following result;
 it  plays the same role Theorem~\ref{T:2'} plays in the proof of Theorem~\ref{T:2}.
\begin{theorem}\label{T:PoliesChar}
Under the  assumptions of Theorem~\ref{T:PoliesConv1},  there exists $k\in \N$, depending only on $\ell$ and the maximum degree of the polynomials $p_1,\ldots, p_\ell$, such that:
If $\nnorm{f_i}_{k,T_i}=0$ for some $i\in \{1,\ldots, \ell\}$, then the averages
\begin{equation}\label{E:Polies}
\frac{1}{N}\sum_{m=1}^N\Big|\frac{1}{b(N)}\sum_{n=1}^{b(N)}    f_1(T_1^{m+p_1(n)}x)\cdot \ldots\cdot f_\ell(T_\ell^{m+p_\ell(n)}x)\Big|^2
\end{equation}
converge pointwise to $0$ as $N\to \infty$.
\end{theorem}
It follows at once that  the factors $\cZ_{k-1,T_{\uepsilon}}$,
defined in Section~\ref{S:Background}, are characteristic factors for pointwise convergence of the averages \eqref{E:Polynomial} and \eqref{E:Polies}.

Using Theorem~\ref{T:PoliesChar}  for  $T_1=\cdots=T_{\ell}=T$, and integrating with respect to $\mu$, we deduce that there exists $k\in\N$ such that  if $\nnorm{f_i}_{k,T}=0$ for some $i\in \{1,\ldots, \ell\}$, then the averages
\eqref{E:rtyu} converge to $0$  in $L^2(\mu)$ as $N\to\infty$. This   recovers one of the main  results from  \cite{Lei05c} needed to prove  convergence in $L^2(\mu)$ for the averages \eqref{E:rtyu}.
\subsection{Open problems related to multiple recurrence}
We state some multiple recurrence problems that are naturally related to the previously established
convergence results. Historically, recurrence problems  have turned out to be easier to establish than
the corresponding convergence problems, but this  does not seem to be the case in our current setup.
\begin{problem}\label{Pr:1}
Let $k\in \N$, $(X,\X,\mu)$ be a probability space,  and for ${\uepsilon}\in V_k$ let $T_{\uepsilon}\colon X\to X$ be measure preserving transformations.
Is it true that   for every $A\in \X$ with $\mu(A)>0$ there exists $\n\in \N^k$
such that
$$
\mu\big(\bigcap_{{\uepsilon}\in V_k} T_{\uepsilon}^{-{\uepsilon}\cdot \n}A\big)>0 \ ?
$$
\end{problem}
We believe that the answer is positive. When all the transformations commute  this is indeed the case. Furthermore, the answer is positive when all the transformations are weak mixing since in this case the corresponding averages converge to $(\mu(A))^{2^k}$ (see \cite{As}, or use  Theorem~\ref{T:2'} in the current article).
In general, even the case $k=2$ is open.
Namely, it is not known whether if $T,S,R$ are measure preserving transformations acting on
the same probability space $(X, \X,\mu)$, and $A\in\X$ satisfies $\mu(A)>0$, then there exist $m,n\in \N$ such that
\begin{equation}\label{E:k=2}
\mu(A\cap T^{-m}A\cap S^{-n}A\cap R^{-m-n}A)>0.
 \end{equation}
 This problem  was first studied by Assani in \cite{As09}. We remark that using Theorem~\ref{T:2'}, one can reduce matters to verifying this multiple recurrence property for very special systems (namely, systems  with  ergodic components  rotations on compact abelian groups), but   we were not able to handle this seemingly simple case.
The non-ergodicity of the transformations causes serious problems and  another obstacle (that becomes more serious in dimension higher than $2$) is that it is not clear why   various approximations arguments that one would like to use preserve the recurrence property \eqref{E:k=2}.
 Interestingly, we were able to overcome the analogous problems for questions pertaining to convergence. Let us also remark that in general no power of $\mu(A)$
 can be used as a lower bound for the multiple intersections in \eqref{E:k=2}. To see this let $S=T^{-2}$,  $R=T^2$  and factor out the transformation $T^{-2n}$; then  the left hand side in
 \eqref{E:k=2} becomes greater than $\mu(A\cap T^{-(m+2n)}A\cap T^{-2(m+2n)}A)$,
and it is known  that in general no  power of $\mu(A)$
 can be used as a lower bound for these expressions (see Theorem~2.1 in \cite{BHK}).
\begin{problem}\label{Pr:2}
Let $\ell\in\N$, $(X,\X,\mu)$ be a probability space, and $T_1,\ldots,T_{\ell}$ be measure preserving transformations acting on $X$.
Furthermore, let $p_1,\ldots,p_\ell$ be distinct polynomials with integer coefficients that satisfy $p_i(0)=0$ for $i=1,\ldots,\ell$.
Is it true that   for every $A\in \X$ with $\mu(A)>0$ there exist $m,n\in \N$ such that
$$
\mu(A\cap T_1^{-m-p_1(n)}A\cap \cdots \cap T_\ell^{-m-p_\ell(n)}A )>0\ ?
$$
\end{problem}
Again, we believe that the answer is positive. Notice that the case where $T_1=\cdots=T_\ell$  corresponds to the so called ``polynomial Szemer\'edi Theorem''
 proved by Bergelson and Leibman~\cite{BL}.
  When all transformations are weak mixing the answer is positive since in this case the corresponding averages converge to $(\mu(A))^{\ell+1}$ (this follows from Theorem~\ref{T:PoliesChar}).
  In general, even the case where all the polynomials are linear is open.
Lastly, let us note that the assumption that the polynomials are distinct is necessary.
It is known (see for example  \cite{BL04}), that there
exist  (non-commuting) transformations $T,S$, acting on the same probability space $(X,\X,\mu)$,  and a set $A\in \X$ with $\mu(A)>0$ and
 such that
$\mu(T^nA\cap S^nA)=0$ for every  $n\in\N$.

\subsection{General conventions and notation}

The following notation will be used throughout the article:  $\N\mathrel{\mathop:}=\{1,2,\ldots\}$, $Tf\mathrel{\mathop:}=f\circ T$,
  $\Re(z)$ is the real part of a complex number $z$. We write $a\colon \Z_N\to \C$ when
  $a\colon \N\to \C$ is a periodic sequence with period $N$.
  We use boldface symbols for vectors.
  If $F$ is a finite set and
$a\colon F\to \C$, then
$\E_{n\in F}a(n)\mathrel{\mathop:}= \frac{1}{|F|}\sum_{n\in F} a(n)$.
 For $r\in \N$,  we denote by
 $S_ra$  the sequence defined by $(S_ra)(n)\mathrel{\mathop:}=a(n+r)$.
We use the symbol $\ll$  when
 some expression is  majorized by a constant multiple of some  other expression. If this constant
  depends on some
 variables $k_1,\ldots, k_\ell$ we write $\ll_{k_1,\ldots, k_\ell}$.

\section{Background Material}\label{S:Background}
We gather  some basic background material that we use throughout this article.
\subsection{Basic facts from ergodic theory}
\subsubsection*{Systems.}
A \emph{system} is a quadruple $(X,\X,\mu,T)$ where
$(X,\X,\mu)$ is a Lebesgue probability space and $T\colon X\to X$ is
an invertible measure preserving transformation.
\subsubsection*{Factors.}
For the context of this article, a \emph{factor} of a system
$(X,\X,\mu, T)$, is a system $(X,\Y,\mu, T)$ where $\Y$ is a
$T$-invariant sub-$\sigma$-algebra of $\X$. We often abuse terminology and
refer to  $\Y$ in place of  the quadruple $(X,\Y,\mu, T)$.


\subsubsection*{Isomorphic systems.} Two systems $(X,\X, \mu, T)$ and $(Y, \Y, \nu, S)$
are {\it isomorphic} if there exists a bijective measurable map $\pi\colon X'\to Y'$,
where $X'$ is a $T$-invariant subset of $X$ and $Y'$ is an
$S$-invariant subset of $Y$, both of full measure, such that
$\mu\circ\pi^{-1} = \nu$ and $(S\circ\pi)(x) = (\pi\circ T)(x)$ for every
 $x\in X'$.



\subsubsection*{Ergodicity and ergodic decomposition.}
We define   $\cI\mathrel{\mathop:}=\{A\in \X\colon \mu(T^{-1}A\triangle A)=0\}$.
A system is \emph{ergodic} if $\cI$ consists only of sets with  measure
  $0$ or $1$. Given an ergodic system and $f\in L^1(\mu)$, the \emph{ergodic theorem}
states that for $\mu$ almost every $x\in X$ we have
$$
\lim_{N\to\infty}\frac{1}{N}\sum_{n=1}^N f(T^nx)=\int f\ d\mu.
$$

Let $x\mapsto\mu_x$ be a regular version of the
conditional measures with respect to the $\sigma$-algebra $\cI$.
This means that the map $x\mapsto\mu_x$ is
$\CI$-measurable, and
for every bounded measurable function $f$ we have
$$
 \E_\mu(f| \cI)(x)=\int f\,d\mu_x\ \text{ for $\mu$ almost every }\ x\in X.
$$
Then the \emph{ergodic decomposition} of $\mu$ is
\begin{equation}\label{E:ErDec}
\mu\mathrel{\mathop:}=\int\mu_x\,d\mu(x).
\end{equation}
The measures $\mu_x$ have the additional property that for $\mu$ almost every $x\in X$ the system $(X,\X,\mu_x,T)$ is ergodic.

\subsection{The seminorms $\nnorm{\cdot}_k$}\label{SS:Seminorms}
The seminorms $\nnorm{\cdot}_k$ were defined for ergodic systems in \cite{HKa}.
These definitions can be easily extended to non-ergodic systems.

Given a system $(X,\X,\mu,T)$ with ergodic decomposition as in  \eqref{E:ErDec}  and a function $f\in L^\infty(\mu)$, we define inductively
\begin{gather}
\label{E:first}
 \nnorm{f}_{1}\mathrel{\mathop:}=\norm{\int f \ d\mu_x}_{L^2(\mu)}\ ;\\
\label{E:recur}
\nnorm f_{k+1}^{2^{k+1}} \mathrel{\mathop:}=\lim_{N\to\infty}\frac{1}{N}\sum_{n=1}^{N}
\nnorm{\bar{f}\cdot T^nf}_{k}^{2^{k}}.
\end{gather}
It can be shown that for every $k\in \N$  the  limit above exists, and $\nnorm{\cdot}_k$, thus defined, is a seminorm on $L^\infty(\mu)$ (see \cite{HKa}, \cite{CFH09}).
If further clarification is needed, we  write  $\nnorm{\cdot}_{k,\mu}$, or $\nnorm{\cdot}_{k,T}$.

More explicitly,  when $k\geq 2$,
one has
\begin{equation}\label{E:semi1}
\nnorm{f}_{k}^{2^k}=
\lim_{N\to\infty} \E_{n_1\in [1,N]} \cdots \lim_{N\to\infty} \E_{n_{k-1}\in [1,N]}  \int  \Big|\int
\prod_{{\uepsilon} \in V_{k-1}}  \mathcal{C}^{|{\uepsilon}|}T^{{\uepsilon}\cdot \n}f
\ d\mu_x\Big|^2d\mu
\end{equation}
where $\n=(n_1,\ldots,n_{k-1})$.
It follows that if $\nnorm{f}_{L^\infty(\mu)}\leq 1$, then
$\nnorm{f}_k\leq \norm{f}_{L^1(\mu)}$ for every $k\in\N$.

For every  function $f\in L^\infty(\mu)$ we have
$$
\nnorm{f}_{k,\mu}^{2^k}=\int \nnorm{f}_{k,\mu_x}^{2^k} \ d\mu(x).
$$
It follows that if $\nnorm{f}_{k,\mu}=0$, then $\nnorm{f}_{k,\mu_x}$ for $\mu$ almost every $x\in X$.

\subsection{Nilsystems and nilsequences}
A \emph{nilmanifold}  is a homogeneous space $X=G/\Gamma$ where  $G$ is a nilpotent Lie group,
and $\Gamma$ is a discrete cocompact subgroup of $G$. If $G_{k+1}=\{e\}$ , where $G_k$ denotes
the $k$-the commutator subgroup of $G$, we say that $X$ is a
 $k$-\emph{step nilmanifold}.

  A $k$-step nilpotent Lie group $G$ acts on $G/\gG$ by left
translation, where the translation by a fixed element $a\in G$ is given
by $T_{a}(g\gG) = (ag) \gG$.  By $m_X$ we denote the unique Borel probability
measure on $X$ that is invariant under the action of $G$ by left
translations (called the {\it normalized Haar measure}), and by $\G/\gG$ we denote the
completion of the Borel $\sigma$-algebra of $G/\gG$. Fixing an element $a\in G$, we call
the system $(G/\gG, \G/\gG, m_X, T_{a})$ a {\it $k$-step  nilsystem}.

If $X=G/\Gamma$ is a $k$-step nilmanifold,  $a\in G$,  $x\in X$, and
$f\in C(X)$,
 we call the sequence $(f(a^nx))_{n\in\N}$ a \emph{basic $k$-step nilsequence}.
A \emph{$k$-step nilsequence}, is a uniform limit of \emph{basic $k$-step nilsequences}.

We are going to use the following result of A.~Leibman (see Theorem~A in \cite{Lei05b}):
\begin{theorem}[{\bf\cite{Lei05b}}] \label{T:ConvNil}
Let $X=G/\Gamma$ be a nilmanifold, $a_1,\ldots,a_\ell\in  G$, $f_1\ldots, f_\ell\in C(X)$, and $p_1,\ldots,p_\ell\colon \Z^d\to \Z$
be polynomials. Then for  every F{\o}lner sequence
  $(\Phi_N)_{N\in\N}$ in $\Z^d$ and $x_1,\ldots,x_\ell \in X$ the averages
  $$
  \frac{1}{|\Phi_N|}\sum_{{\bf n}\in \Phi_N}f_1(a_1^{p_1({\bf n})}x_1) \cdots  f_\ell(a_\ell^{p_\ell({\bf n})}x_\ell)
$$
converge as $N\to \infty$.
\end{theorem}

\subsection{The factors $\CZ_k$ and their structure} \label{SS:Z_k}
Given a system $(X,\cX,\mu,T)$, it was shown in \cite{HKa} (for ergodic systems but the same construction works for general systems) that
for every $k\geq 1$, there exists a $T$-invariant sub-$\sigma$-algebra $\cZ_{k-1}$ of $\X$
that satisfies
\begin{equation}
\label{E:DefZ_l}
\text{\em for } \ f\in L^\infty(\mu),\ \ \E(f|\cZ_{k-1})=0\quad \text{ \em if and
  only if } \quad   \nnorm f_{k,T} = 0.
\end{equation}
The connection between the factors $\cZ_k$ of a given  system and
nilsystems is given by the following structure theorem of Host and Kra:
\begin{theorem}[{\bf\cite{HKa}}]\label{T:Structure}
  Let $k\in \N$ and $(X,\X,\mu,T)$ be a system  with ergodic decomposition as in \eqref{E:ErDec}.
  Then  for $\mu$ almost every $x\in X$ the  system  $(X,\mathcal{Z}_{k},\mu_x,T)$ is an inverse limit of  $k$-step
  nilsystems.
\end{theorem}
The conclusion in the preceding statement means that for $\mu$ almost every $x\in X$
 for a given measure $\mu_x$ there exists  an increasing  sequence of $T$-invariant sub-$\sigma$-algebras
$(\X_j)_{j\in\mathbb{N}}$ (depending on $\mu_x$),   such that
$\bigvee_{j\in\N}\X_j=\X$ up to sets of $\mu_x$-measure
zero, and each system $(X,\X_j,\mu_x,T)$ is isomorphic to a $k$-step
  nilsystem.

We remark that although we do not make explicit use of Theorem~\ref{T:Structure} in this article,
it is  a key ingredient
in the proof of Proposition~\ref{P:ApprNil} that is crucial for our analysis.

\section{Characteristic factors and convergence for cubic averages}
\subsection{Characteristic factors for cubic averages} We are going to prove Theorem~\ref{T:2'}.
 The main idea is best  illustrated by considering the
case of cubic averages of dimension $2$.
Assuming for example that  $f_1\in L^\infty(\mu)$ satisfies $\nnorm{f_1}_{2,\mu,T_1}=0$, and $f_2,f_3\in L^\infty(\mu)$, our goal is to establish the pointwise identity
$$
\lim_{N\to\infty}|\E_{ m,n\in  [1,N]} f_1(T_1^mx)\cdot f_2(T_2^nx)\cdot f_3(T_3^{m+n}x)|=0.
$$
It suffices to show that for $\mu$ almost every $x\in X$ we have
\begin{equation}\label{E:a1}
\lim_{N\to \infty}\E_{ n\in [1,N]}|\E_{m\in [1,N]} f_1(T_1^mx)\cdot f_3(T_3^{m+n}x)|^2=0.
\end{equation}
Using suitable applications of a variation of van der Corput's fundamental lemma (the precise statement is given in Lemma~\ref{L:VDC})
we can show (see Proposition~\ref{P:key}) that  the limit in \eqref{E:a1} is bounded by a constant
 multiple of
\begin{equation}\label{E:a2}
\lim_{N\to\infty} \E_{n\in [1,N]} |\lim_{N\to\infty} \E_{m\in [1,N]}  \bar{f}_1(T_1^mx)\cdot f_1(T_1^{m+n}x) |^2.
\end{equation}
The ergodic theorem implies that for $\mu$ almost every $x\in X$ the last limit is
equal to
$$
\lim_{N\to\infty} \E_{n\in [1,N]}  \Big|\int \bar{f}_1(x)\cdot f_1(T_1^nx) \ d\mu_{x,T_1} \Big|^2=\nnorm{f_1}_{2,\mu_x,T_1}^4
$$
where $\mu=\int
\mu_{x,T_1} \ d\mu(x)$ is the ergodic decomposition of the measure $\mu$ with respect to $T_1$.
 Since $\nnorm{f_1}_{2,\mu,T_1}=0$ implies that  $\nnorm{f_1}_{2,\mu_x,T_1}=0$ for $\mu$ almost every $x\in X$,
 our goal is established.

Since most of the calculations and estimates do not depend on the dynamical structure
of the sequences $(f_{\uepsilon}(T_{\uepsilon}^nx))_{n\in \N}$ (it is only at the very last step of
the argument that we use the pointwise ergodic theorem to take advantage of this extra structure)
we work them out for general bounded sequences $(a_{\uepsilon}(n))_{n\in\N}$.

Key to our study will be some quantities  that control the limiting behavior of the cubic averages \eqref{E:A_k} when the
sequences $(f_{\uepsilon}(T_{\uepsilon}^nx))_{n\in\N}$ are replaced by general bounded sequences $(a_{\uepsilon}(n))_{n\in\N}$.
Closely related quantities have been defined by T.~Gowers in \cite{Gow01} and by B.~Host and B.~Kra in \cite{HK09}.
We define these and prove some basic estimates in the next subsections.
\subsubsection{Measures of uniformity}
We remind the reader that when we write $b\colon \Z_N\to \C$ we refer to a periodic sequence
  $b\colon \N\to \C$ with period $N$.
We say that $a=(a_N)_{N\in \N}$, where $a_N\colon \Z_N\to \C$, is \emph{uniformly bounded}, if
 there exists a constant $C\in \R$ such that $|a_N(n)|\leq C$ for every $n\in [1,N]$ and $N\in\N$.
For $k\in\N$,  $z\in \C$, and ${\uepsilon}\in V_{k}$,  we let
$|{\uepsilon}|\mathrel{\mathop:}=\epsilon_1+\cdots+\epsilon_k$, and  $ \mathcal{C}^kz\mathrel{\mathop:}=z$ if $k$ is even, and
 $ \mathcal{C}^k z\mathrel{\mathop:}=\bar{z}$ if $k$ is odd.

We let
$$
\nnorm{a}_{U_{1}(\N)}\mathrel{\mathop:}=\limsup_{N\to\infty}|\E_{n\in [1,N]}a_N(n)|,
$$
and for $k\geq 2$  we define
\begin{multline}\label{E:sdc}
\nnorm{a}_{U_k(\N)}\mathrel{\mathop:}=\\
\Big(\limsup_{N\to\infty}\E_{ n_{1}\in [1,N]}\cdots \limsup_{N\to\infty}\E_{n_{k-1}\in [1,N]} \limsup_{N\to\infty}\big|\E_{m\in  [1,N]}
 \prod_{{\uepsilon} \in V_{k-1}}  \mathcal{C}^{|{\uepsilon}|} a_N(m+{\uepsilon} \cdot \n)\big|^2\Big)^{\frac{1}{2^k}}
\end{multline}
where $\n=(n_1,\ldots,n_{k-1})$.

Furthermore, for  $N\in \N$ we let
$$
\nnorm{a_N}_{U_{1}(\Z_N)}\mathrel{\mathop:}=|\E_{ n\in [1,N]}a_N(n)|,
$$
and for $k\geq 2$ we define
$$
\nnorm{a_N}_{U_k(\Z_N)}\mathrel{\mathop:}=
\big(\E_{\n\in [1,N]^{k-1}} \big|\E_{m\in [1,N]} \prod_{{\uepsilon} \in V_{k-1}}  \mathcal{C}^{|{\uepsilon}|} a_N(m+{\uepsilon}\cdot  \n)\big|^2\big)^{\frac{1}{2^k}}.
$$
This is the so called  Gowers norm of $a_N$.

Given a bounded sequence $a\colon \N\to \C$ ,  for $N\in \N$, we define  $a_N\colon \Z_N\to \C$  by $a_{N}(n+N\Z)\mathrel{\mathop:}=a(n)$ for $ n\in [1, N]$. We  let $\tilde{a}\mathrel{\mathop:}=(a_N)_{N\in\N}$.
Furthermore, we define
$$
\nnorm{a}_{U_k(\N)}\mathrel{\mathop:}=\nnorm{\tilde{a}}_{U_k(\N)},\quad   \nnorm{a}_{U_k(\Z_N)}\mathrel{\mathop:}=\nnorm{a_N}_{U_k(\Z_N)}.
$$
Notice that $\nnorm{a}_{U_k(\N)}$ can also be computed by replacing $a_N$ with $a$ in \eqref{E:sdc}.

One immediately sees that
 $\nnorm{\cdot}_{U_k(\N)}$    satisfies  the recursive identity
\begin{equation}\label{E:inductive}
\limsup_{N\to\infty}\E_{r\in [1,N]}\nnorm{S_ra \cdot\bar{a}}_{U_k(\N)}^{2^{k}}=\nnorm{a}_{U_{k+1}(\N)}^{2^{k+1}}.
\end{equation}
We caution the reader that the triangle inequality does not necessarily hold for
$\nnorm{\cdot}_{U_k(\N)}$,  but this is not going to play any role in this article.

The next result   links the seminorms $\nnorm{\cdot}_{U_k(\N)}$
with the ergodic seminorms $\nnorm{\cdot}_k$ that were defined in Section~\ref{SS:Seminorms} (a similar result
was also established in Corollary~3.11 of \cite{HK09}).
\begin{proposition}\label{P:RelToErgodic}
Let $(X,\mathcal{X},\mu,T)$ be a measure preserving system
with ergodic decomposition $\mu=\int \mu_x \ d\mu(x)$ and $f\in L^\infty(\mu)$. Then
  for $\mu$ almost every  $x\in X$  we have
$$
\nnorm{f(T^nx)}_{U_k(\N)}= \nnorm{f}_{k,\mu_x}.
$$
\end{proposition}
\begin{proof}
The ergodic theorem gives that for  $\mu$ almost every  $x\in X$, for every $\n\in  \N^{k-1}$ and ${\uepsilon}\in V_{k-1}$ we have
$$
\lim_{N\to\infty}\E_{ m\in [1, N]}
 \prod_{{\uepsilon} \in V_{k-1}} \mathcal{C}^{|{\uepsilon}|}f(T^{m+{\uepsilon}\cdot \n }x)=
\int
\prod_{{\uepsilon} \in V_{k-1}} \mathcal{C}^{|{\uepsilon}|}T^{{\uepsilon}\cdot \n}f
\ d\mu_{x}.
$$
The result now follows by using   the definition of   $\nnorm{\cdot}_{U_k(\N)}$ and  formula $\eqref{E:semi1}$.
\end{proof}

\subsubsection{Comparing $\nnorm{\cdot}_{U_k(\Z_\infty)}$ with  $\limsup_{N\to\infty}\nnorm{\cdot}_{U_k(\Z_N)}$. }
The following estimate will be key for our analysis:
\begin{proposition}\label{P:U_k}
Let $a=(a_N)_{N\in\N}$, where $a_N\colon \Z_N\to \C$,  be uniformly bounded by $1$.
Then for every $k\in\N$ we have
$$
\limsup_{N\to\infty}\nnorm{a}_{U_k(\Z_N)} \ll_k  \nnorm{a}_{U_k(\N)}.
$$
\end{proposition}
To prove Proposition~\ref{P:U_k} we are going to use the following variation of van der Corput's
fundamental lemma:
\begin{lemma}\label{L:VDC}
Let  $N\in\N$ and  $a\colon \Z_N\to \mathbb{C}$.
Then for every  $R\in \N$ we have
$$
|\E_{n\in [1,N]} a(n)|^2\leq 2\cdot \E_{ r\in [1, R]} \Big(1-\frac{r}{R}\Big)\Re\big(\E_{n\in [1,N]}a(n+r)\cdot\bar{a}(n))
+ \frac{\E_{n\in [1,N] } |a(n)|^2}{R}.
$$
\end{lemma}
\begin{proof}
Let $R\in \N$.
Using the identity
$$
\E_{n\in [1,N]} a(n)= \E_{n\in [1,N]}\E_{ r\in [1, R]} a(n+r)
$$
and the Cauchy-Schwarz inequality, we get that $|\E_{n\in [1,N]} a(n)|^2$ is bounded by
$$
\E_{n\in [1,N]}|\E_{r\in [1, R]} a(n+r)|^2=\E_{ r,r' \in [1,R]} \E_{n\in [1,N]} a(n+r) \cdot \bar{a}(n+r').
$$
Isolating those terms for which $r=r'$, and using the symmetry up to conjugation
of the remaining expression with respect to $r$ and $r'$,
  we see that the last expression  is equal to
 $$
 \frac{2}{R^2}\cdot \sum_{1\leq r'<r\leq R} \Re\big(\E_{n\in [1,N]} a(n+r) \cdot \bar{a}(n+r')\big) + \frac{\E_{n\in [1,N]}|a(n)|^2}{R}.
 $$
 To end the proof, it suffices to perform the change of variables $n\to n-r'$
 and notice that  for $k\in \{1,\ldots, R\}$  the equation
$r-r'=k$ with $1\leq r'<r\leq R$ has $R-k$ solutions.
\end{proof}

\begin{lemma}\label{L:U_k}
Let  $N\in\N$ and  $a\colon \Z_N\to \mathbb{C}$ be  bounded by $1$.
Then for
every  $R\in \N$ we have
$$
\E_{n\in [1,N]}|\E_{m\in [1,N]}a(m+n)\cdot \bar{a}(m)|^2\leq 2\cdot  \E_{r\in [1, R]}
|\E_{m\in [1,N]} a(m+r)\cdot \bar{a}(m)|^2 +1/R.
$$
\end{lemma}
\begin{proof}
Using  Lemma~\ref{L:VDC}
we deduce that  the left hand side is bounded by
$$
2\cdot \E_{n\in [1,N]}\E_{ r\in [1, R]} \Big(1-\frac{r}{R}\Big)\Re\big( \E_{m\in  [1,N]}a(m+n+r)\cdot \bar{a}(m+r)\cdot  \bar{a}(m+n)\cdot a(m)\big)
+1/R.
$$
Interchanging the averages and performing the change of variables $n\to n-m$ we deduce that the last expression is equal to
$$
2\cdot \E_{r\in [1, R]} \Big(1-\frac{r}{R}\Big)|\E_{m\in  [1,N]}a(m+r)\cdot   \bar{a}(m)  |^2
+1/R.
$$
The result follows.
\end{proof}
Next we  prove Proposition~\ref{P:U_k} by successively applying  Lemma~\ref{L:U_k}.
\begin{proof}[Proof of Proposition~\ref{P:U_k}]
Remember  that $a_N\colon \Z_N\to \C$  is defined by
$a_{N}(n+N\Z)\mathrel{\mathop:}=a(n)$ for $n\in [1,N] $.
For $k=1$ we have
$$
\limsup_{N\to\infty} \nnorm{a}_{U_1(\Z_N)} = \limsup_{N\to\infty} \nnorm{a_N}_{U_1(\Z_N)} =  \nnorm{a}_{U_1(\N)}.
$$

 Suppose that the statement holds for
$k\in \N$, we are going to show that it holds for  $k+1$.
We have
\begin{equation}\label{E:a1'}
\nnorm{a_N}^{2^{k+1}}_{U_{k+1}(\Z_N)}=\E_{n_1,\ldots,n_k\in  [1,N]}\big|\E_{m\in [1,N]}
\prod_{{\uepsilon} \in V_{k}}
 \mathcal{C}^{|{\uepsilon}|} a_N(m+\epsilon_1n_1+\cdots+\epsilon_{k}n_{k})\big|^2.
\end{equation}
 We fix $n_1,\ldots,n_{k-1}\in  [1,N]$, and apply Lemma~\ref{L:U_k} for
$A_{N,n_1,\ldots,n_{k-1}}\colon \Z_N\to \C$ defined by
$$A_{N,n_1,\ldots,n_{k-1}}(m)=\prod_{{\uepsilon} \in V_{k-1}}  \mathcal{C}^{|{\uepsilon}|}a_N(m+\epsilon_1n_1+\cdots+\epsilon_{k-1}n_{k-1}).$$
We deduce that for every
$R,N\in \N$, the right hand side of \eqref{E:a1'} is  bounded by $2$ times
\begin{equation*}
\E_{n_1,\ldots,n_{k-1}\in [1,N]}\E_{ n_k\in [1,  R]}\big|\E_{m\in [1,N]} \prod_{{\uepsilon} \in V_{k}} \mathcal{C}^{|{\uepsilon}|}a_N(m+\epsilon_1n_1+\cdots+\epsilon_{k}n_{k})\big|^2 +1/R.
\end{equation*}
Next, we fix $n_k \in \N$, and use the inductive hypothesis  for the sequence
$
S_{n_k}a\cdot \bar{a}$.
We get
\begin{multline*}
\limsup_{N\to \infty}\E_{n_1,\ldots,n_{k-1}\in [1,N]}\Big|\E_{m\in [1,N]}
\prod_{{\uepsilon} \in V_{k}}  \mathcal{C}^{|{\uepsilon}|}a_N(m+\epsilon_1n_1+\cdots+\epsilon_{k}n_{k})\Big|^2=\\
\limsup_{N\to\infty} \nnorm{S_{n_k}a\cdot \bar{a}}^{2^{k}}_{U_{k}(\Z_N)}\ll_k
\nnorm{S_{n_k} a\cdot \bar{a}}^{2^k}_{U_k(\N)}.
\end{multline*}
Combining the previous estimates
we get for every
positive integer $R$ that
$$
\limsup_{N\to\infty} \nnorm{a}^{2^{k+1}}_{U_{k+1}(\Z_N)}=\limsup_{N\to\infty} \nnorm{a_N}^{2^{k+1}}_{U_{k+1}(\Z_N)}
\ll_k
\E_{ n_k\in[1, R]}\nnorm{S_{n_k} a\cdot \bar{a}}_{U_k(\N)}^{2^k}+1/R.
$$
Finally, taking the $\limsup$ as  $R\to\infty$,  and using  the identity
\eqref{E:inductive}
we get the advertised estimate.
\end{proof}

\subsubsection{Proof of Theorem~\ref{T:2'}}
We first recall a known estimate (see Lemma~3.8 in \cite{Gow01}).
 \begin{lemma}[{\bf Gowers-Cauchy-Schwarz Inequality}]\label{L:GCS}
Let $k\geq 2$ be an integer, $N\in \N$, and for ${\uepsilon}\in V_{k-1}$ let $a_{\uepsilon}\colon \Z_N\to \C$.
Then
$$
\E_{\n\in [1,N]^{k-1}}\big|\E_{m\in [1,N]}
\prod_{{\uepsilon} \in V_{k-1}}  \mathcal{C}^{|{\uepsilon}|} a_{\uepsilon}(m+{\uepsilon}\cdot \n)
\big|^2
\leq
\prod_{{\uepsilon}\in V_{k-1}}\nnorm{a_{\uepsilon}}^2_{U_k(\Z_N)}.
$$
\end{lemma}
Combining Lemma~\ref{L:GCS} and  Proposition~\ref{P:U_k} we are going to prove the following key estimate:
\begin{proposition}\label{P:key}
Let $k\geq 2$ be an integer  and for ${\uepsilon}\in V_{k-1}$ let  $a_{\uepsilon}\colon \N\to \C$
be sequences. Then
\begin{equation}\label{E:bas}
\limsup_{N\to\infty}\E_{\n\in [1,N]^{k-1}} \big|\E_{m\in [1,N]}
 \prod_{{\uepsilon} \in V_{k-1}}
  \mathcal{C}^{|{\uepsilon}|} a_{\uepsilon}(m+{\uepsilon}\cdot \n) \big|^2 \ll_k
 \prod_{{\uepsilon}\in V_{k-1}}\nnorm{a_{\uepsilon}}^2_{U_k(\N)}.
\end{equation}
\end{proposition}
\begin{proof}
We fix $k\geq 2$, $N\in \N$, and for  ${\uepsilon}\in V_{k-1}$ we  define $a_{{\uepsilon},N}\colon \Z_{N}\to \C$ as follows:
$$
a_{\uepsilon,N}(n+N\Z)=\begin{cases}a_{{\bf 0}}(n) \cdot {\bf 1}_{[1,[N/k]]}(n)  & \quad \text{ for } \ {\uepsilon} = {\bf 0} \\
a_{\uepsilon}(n)   & \quad \text{ for } \ {\uepsilon} \neq {\bf 0}
\end{cases}$$
where ${\bf 0}=(0,0,\ldots, 0)$ and $n\in \{1,\ldots, N\}$.
Suppose that the element  $\n\in \N^{k-1}$ has all its coordinates in the interval $[1, [N/k]]$. Then
$$
\prod_{{\uepsilon} \in V_{k-1}}  \mathcal{C}^{|{\uepsilon}|} a_{{\uepsilon},N}(m+{\uepsilon}\cdot \n)
=\begin{cases}
\prod_{{\uepsilon} \in V_{k-1}}  \mathcal{C}^{|{\uepsilon}|} a_{\uepsilon}(m+{\uepsilon}\cdot \n) & \quad \text{for }\
m\in [1, [N/k]]\\
0 & \quad \text{for }\
m\in ([N/k],N].
\end{cases}
$$
It follows that
$$
\E_{\n\in [1, [N/k]]^{k-1}} \big|\E_{ m\in [1, [N/k]]}
\prod_{{\uepsilon} \in V_{k-1}}  \mathcal{C}^{|{\uepsilon}|}a_{\uepsilon}(m+{\uepsilon}\cdot \n)
\big|^2
$$
is at most
$$
k\cdot \E_{\n\in [1, [N/k]]^{k-1}} \big|\E_{m\in [1,N]} \prod_{{\uepsilon} \in V_{k-1}}  \mathcal{C}^{|{\uepsilon}|} a_{{\uepsilon},N}(m+{\uepsilon}\cdot \n)\big|^2,
$$
which in turn is at most
$$
k^k\cdot \E_{\n\in [1,N]^{k-1}} \big|\E_{m\in [1,N]} \prod_{{\uepsilon} \in V_{k-1}}  \mathcal{C}^{|{\uepsilon}|}a_{{\uepsilon},N}(m+{\uepsilon}\cdot \n)\big|^2.
$$
Using Lemma~\ref{L:GCS} we see that the last expression  is bounded by a constant multiple of
$$
\prod_{{\uepsilon}\in V_{k-1}}\nnorm{a_{{\uepsilon},N}}_{U_k(\Z_{N})}^{2}.
$$
Combining the above,  taking limits as $N\to \infty$, and  using Proposition~\ref{P:U_k}, we deduce that the left hand side of
\eqref{E:bas} is bounded by a constant multiple of
$$
 \prod_{{\uepsilon}\in V_{k-1}}\nnorm{\tilde{a}_{{\uepsilon}}}^{2}_{U_k(\N)}
$$
where $\tilde{a}_{{\uepsilon}}=(a_{{\uepsilon},N})_{N\in \N}$.
Furthermore, an easy computation shows that
$$
\nnorm{\tilde{a}_{\uepsilon}}_{U_k(\N)}=
\begin{cases}
k^{-1}\cdot \nnorm{a_{\bf 0}}_{U_k(\N)}, &\quad  \text{ for }  \ \uepsilon={\bf 0}\\  \nnorm{a_{\uepsilon}}_{U_k(\N)},&  \quad  \text{ for } \ {\uepsilon}\neq {\bf 0}.
\end{cases}
$$
This completes the proof.
\end{proof}
Applying the previous estimate  for suitably chosen sequences we get the following:
\begin{corollary}\label{C:key'}
Let $k\geq 2$ be an integer, $(X,\X,\mu)$ be a probability space,  and for ${\uepsilon}\in V_{k-1}$ let $T_{\uepsilon}\colon X\to X$ be measure preserving transformations, and $f_{\uepsilon}\in L^\infty(\mu)$ be functions. Furthermore, let $\mu=\int \mu_{x,T_{\uepsilon}} \ d\mu(x)$ be the ergodic decomposition of the measure $\mu$ with respect to $T_{\uepsilon}$.
Then for $\mu$ almost every $x\in X$ we have
$$
\limsup_{N\to\infty}
\E_{\n\in [1,N]^{k-1}}  \big|\E_{m\in [1, N]}
 \prod_{{\uepsilon} \in V_{k-1}} f_{{\uepsilon}}(T_{\uepsilon}^{m+{\uepsilon}\cdot \n}x)
\big|^2\ll_k
 \prod_{{\uepsilon}\in V_{k-1}} \nnorm{f_{\uepsilon}}^{2}_{k,\mu_{x, T_{\uepsilon}}}.
$$
\end{corollary}
\begin{proof}
Let 
$x\in X$.
Applying Proposition~\ref{P:key} for the sequences $a_{\uepsilon}(n)=f_{{\uepsilon}}(T_{{\uepsilon}}^nx)$, ${\uepsilon} \in V_{k-1}$, we get that the left hand side
is bounded by a constant multiple of
$$
\prod_{{\uepsilon}\in V_{k-1}} \nnorm{f_{{\uepsilon}}(T_{{\uepsilon}}^nx)}^{2}_{U_k(\N)}.
$$
Proposition~\ref{P:RelToErgodic}
gives that
  for  every ${\uepsilon}\in  V_{k-1}$,
  for $\mu$ almost every  $x\in X$, we have
$$
\nnorm{f_{\uepsilon}(T_{\uepsilon}^nx)}_{U_k(\N)}= \nnorm{f_{\uepsilon}}_{k,\mu_{x,T_{{\uepsilon}}}}.
$$
This completes the proof.
\end{proof}
We are now one small step from proving  Theorem~\ref{T:2'}.
\begin{proof}[Proof of Theorem~\ref{T:2'}]
Suppose that  $\norm{f_{\bf 1}}_{k,\mu,T_{\bf 1}}=0$, where ${\bf 1}=(1,0,\cdots, 0)$. The proof is similar in the other cases.
We want to show that for almost every $x\in X$
$$
\lim_{N\to\infty}\E_{\n\in [1,N]^{k}}
  \prod_{{\uepsilon} \in V^*_{k}} f_{{\uepsilon}}(T_{\uepsilon}^{{\uepsilon}\cdot \n}x)
  =0.
$$
Using the Cauchy-Schwarz inequality
and bounding
 all functions $f_{(0,\uepsilon)}$, where $\uepsilon\in V_{k-1}$, by their sup norm,
we deduce that the expression
$$
\big|\E_{\n\in [1,N]^{k}}
\prod_{{\uepsilon} \in V^*_{k}} f_{{\uepsilon}}(T_{\uepsilon}^{{\uepsilon}\cdot \n}x)
\big|^2
$$
is bounded by a constant multiple of an average of the form
\begin{equation}\label{E:dfg}
\E_{\n\in [1,N]^{k-1}}
\big|\E_{m\in [1, N]}  \prod_{{\uepsilon} \in V_{k-1}} \tilde{f}_{\uepsilon}(\tilde{T}_{\uepsilon}^{m+{\uepsilon}\cdot \n}x)
\big|^2,
\end{equation}
where $\tilde{f}_{\bf 1}=f_{\bf 1}$,  $\tilde{f}_{\uepsilon}\in \{f_{(1,\uepsilon)}, {\uepsilon} \in V^*_{k-1}\}$,  and $\tilde{T}_{\uepsilon}\in \{T_{(1,\uepsilon)}, {\uepsilon}\in V^*_{k-1}\}$ for ${\uepsilon}\in V^*_{k-1}$.

 Since $\nnorm{f_{\bf 1}}_{k,\mu,T_{\bf 1}}=0$ implies that $\nnorm{f_{\bf 1}}_{k,\mu_{x,T_{\bf 1}}}=0$ for $\mu$ almost every  $x\in X$, Corollary~\ref{C:key'} gives that
for $\mu$ almost every $x\in X$ the averages \eqref{E:dfg} converge to $0$. This  completes the   proof.
\end{proof}

\subsection{Convergence of cubic averages}\label{SS:2.2}
 In this section we are going to prove Theorem~\ref{T:2}.
A natural approach for establishing such a convergence result
   would be to try to combine  Theorem~\ref{T:2'} with   Theorem~\ref{T:Structure}, in order  to reduce matters to the case where all systems are nilsystems. Such an approach works well
 when all the transformations are equal, but in our more  general setup it presents problems
 that are difficult to circumvent. For instance, although it is possible to reduce matters to the case where for every $\uepsilon \in V_k^*$ the ergodic components of the transformation $T_{\uepsilon}$ are inverse limits of nilsystems, the various ergodic disintegrations and sub-$\sigma$-algebras involved in the inverse limits cannot be taken to be the same for each transformations $T_{\uepsilon}$ (even if the transformations commute).
To overcome this problem we
work pointwise, and use an approach similar to the one  in \cite{CFH09}.  We combine Theorem~\ref{T:2} with a pointwise  decomposition result  that applies to general (not necessarily ergodic) systems. It is a direct consequence of Proposition~3.1 from \cite{CFH09}
which in turn is a non-trivial
  consequence of the structure theorem of Host and Kra stated in  Theorem~\ref{T:Structure}.
\begin{proposition} \label{P:ApprNil}
Let $(X,\X,\mu,T)$ be a system, $f\in L^\infty(\mu)$,  and $k\in \N$.
Then for every $\varepsilon>0$, there exist measurable functions $f^s, f^u, f^e,$
with $L^\infty(\mu)$ norm at most $2\norm{f}_{L^\infty(\mu)}$, such that
\begin{enumerate}
\item $f=f^s+f^u+f^e$;
\item $\nnorm{f^u}_{k+1}=0$; \  $\norm{f^e}_{L^1(\mu)}\leq \varepsilon$; \ and
\item for $\mu$ almost every $x\in X$, the sequence $(f^s (T^nx))_{n\in\N}$
is a $k$-step nilsequence.
\end{enumerate}
\end{proposition}
Arithmetic versions of this result
were recently established  in \cite{GT10a} and in \cite{Sz10}.
  The reader is advised to think of the function $f^e$ as an error term; when one works with convergence problems it  typically
 can  be shown to have a negligible effect on our averages (but this is not the case  for recurrence problems unless
 one aims at a uniform lower bound). The function $f^u$ is
  the uniform component and it too can be neglected once the appropriate uniformity estimates are obtained. Finally,
   the function $f^s$ is the structured
  component; this has to be further analyzed, typically using equidistribution results on nilmanifolds.



\begin{proof}[Proof of Theorem~\ref{T:2}]
For $k=1$ the result follows from the pointwise ergodic theorem. So we can assume that $k\geq 2$.
Furthermore, we can assume that  $\norm{f_{\uepsilon}}_{L^\infty(\mu)}\leq 1$ for every ${\uepsilon} \in V^*_{k}$.
Let
$$
A_N(f_{\uepsilon})(x)\mathrel{\mathop:}=
\E_{\n\in [1,N]^{k}}
\prod_{{\uepsilon} \in V^*_{k}} f_{{\uepsilon}}(T_{\uepsilon}^{{\uepsilon}\cdot \n}x).
$$
We are going to show that for $\mu$ almost every $x\in X$ the sequence $(A_N(f_{\uepsilon})(x))_{N\in\N}$
is Cauchy.

By  Proposition~\ref{P:ApprNil}, we have that for every  $m\in\N$ and ${\uepsilon}\in V^*_k$, there exist measurable functions
$f^s_{{\uepsilon},m}, f^u_{{\uepsilon},m}, f^e_{{\uepsilon},m}$,  with  $L^\infty(\mu)$ norm bounded by $2$, and such that
\begin{enumerate}
\item $f_{{\uepsilon}}=f^s_{{\uepsilon},m}+ f^u_{{\uepsilon},m}+ f^e_{{\uepsilon},m}$;
\item \label{E:ue} $\nnorm{f^u_{{\uepsilon},m}}_{k,T_{\uepsilon}}=0$;  \ $\norm{f^e_{{\uepsilon},m}}_{L^1(\mu)}\leq 1/m$; \ and
\item \label{E:s} for $\mu$ almost every $x\in X$, the sequence $(f^s_{{\uepsilon},m} (T_{\uepsilon}^nx))_{n\in\N}$
is a $(k-1)$-step nilsequence.
\end{enumerate}

First we study the contribution of the functions $f^u_{{\uepsilon},m}$. Combining  property \eqref{E:ue} with  Theorem~\ref{T:2'}, we see that
when evaluating the limit of the averages $A_N(f_{\uepsilon})$, we  can ignore the contribution of these functions, namely,
for every $m\in\N$,  for $\mu$ almost every $x\in X$ we have
\begin{equation}\label{E:AN1}
\lim_{N\to\infty}|A_N(f_{{\uepsilon}})(x) -A_N(f^s_{{\uepsilon},m}+f^e_{{\uepsilon},m})(x) |=0.
\end{equation}

Next, we study the contribution of the functions $f^e_{{\uepsilon},m}$. We are going to show that  this too is essentially negligible,
as long as we consider suitably large values of $m$. Indeed, if we expand the expression $A_N(f^s_{{\uepsilon},m}+f^e_{{\uepsilon},m})-A_N(f^s_{{\uepsilon},m})$,
 use Corollary~\ref{C:key'} to bound each of the terms, and also use that
$ \nnorm{f^e_{{\uepsilon},m}}_{k,\mu_{x,T_{{\uepsilon}}}}\leq  2$ and
$\nnorm{f^s_{{\uepsilon},m}}_{k,\mu_{x,T_{{\uepsilon}}}}\leq 2$, we get
 for $\mu$ almost every $x\in X$ the bound
$$
\limsup_{N\to\infty}|A_N(f^s_{{\uepsilon},m}+f^e_{{\uepsilon},m})(x)-
A_N(f^s_{{\uepsilon},m})(x)|\ll_k
\max_{{\uepsilon}\in V_k^*} \nnorm{f^e_{{\uepsilon},m}}_{k,\mu_{x,T_{{\uepsilon}}}}
\ll \max_{{\uepsilon}\in V_k^*} \norm{f^e_{{\uepsilon},m}}_{L^1(\mu_{x,T_{{\uepsilon}}})}.
$$
By property $(ii)$ we have  $\lim_{m\to\infty}\int|f^e_{{\uepsilon},m}|\ d\mu=0$
for ${\uepsilon}\in V^*_k$., and as a consequence  there exists a  sequence $(m_l)_{l\in\N}$, with
$m_l\to \infty$, and such that for $\mu$ almost every $x\in X$ we have
$$
\lim_{l\to\infty}\int|f^e_{{\uepsilon},m_l}|\ d\mu_{x,T_{\uepsilon}}=0
$$
 for every ${\uepsilon}\in V^*_k$. From the preceding discussion it follows that for $\mu$ almost every
 $x\in X$
\begin{equation}\label{E:AN2}
\lim_{l\to\infty}\limsup_{N\to\infty}|A_N(f^s_{{\uepsilon},m_l}+f^e_{{\uepsilon},m_l})(x) -A_N(f^s_{{\uepsilon},m_l})(x)|=0.
\end{equation}

Combining \eqref{E:AN1} and \eqref{E:AN2} we get for $\mu$ almost every $x\in X$ that
\begin{equation}\label{E:iou}
\lim_{l\to\infty}\limsup_{N\to\infty}|A_N(f_{\uepsilon})(x)
-A_N(f^s_{{\uepsilon},m_l})(x)|=0.
\end{equation}
Since  by property $(iii)$, for $\mu$ almost every $x\in X$, for every $l\in \N$, and $\uepsilon\in V_k^*$, the sequence $(f^s_{{\uepsilon},m_l} (T_{\uepsilon}^nx))_{n\in\N}$
is a  nilsequence,  it follows  from Theorem~\ref{T:ConvNil} that for  $\mu$ almost every $x\in X$, for every $l\in \N$,
the averages $A_N( f^s_{{\uepsilon},m_l})(x)$ converge.
Combining this with \eqref{E:iou}, we deduce that for  $\mu$ almost every $x\in X$ the sequence $(A_N(f_{\uepsilon})(x))_{N\in\N}$ is Cauchy. This  completes the proof.
\end{proof}

\section{Characteristic factors and convergence for polynomial  averages}
In this section we are going to prove Theorems~\ref{T:PoliesConv1} and \ref{T:PoliesChar}. As was the case with the cubic
averages, some uniformity estimates for general bounded sequences play a key role in the argument.
We start with establishing these.
\subsection{Uniformity estimates}
We remind the reader that
in the forthcoming statements $b\colon\N\to\N$ is a sequence that satisfies
$$
b(N)\to \infty \quad  \text{ and } \quad b(N)/N^{1/d}\to 0
$$ where  $d$ is the maximum degree of the polynomials involved
in each statement. To avoid confusion, let  us also remark that 
none of the sequences defined defined in this section is
assumed to be periodic.

Our goal in this section is to establish the following estimate:
\begin{proposition}\label{P:polies}
Let $a_1,\ldots, a_\ell\colon \N\to \C$ be bounded sequences and  $p_1,\ldots,p_\ell\in \Z[t]$ be non-constant polynomials such that   $p_i-p_j$ is non-constant for $i\neq j$.
Then there exists $k\in \N$, depending only on $\ell$  and the maximum degree of the polynomials $p_1,\ldots,p_\ell$, such that if  $\nnorm{a_i}_{U_{k}(\N)}=0$ for some $i\in \{1,\ldots,\ell\}$, then
$$
\lim_{N\to\infty}  \E_{ m\in [1,N]}\Big|\E_{ n\in [1, b(N)]}
\prod_{i=1}^\ell   a_i(m+p_i(n))\Big|^2=0.
$$
\end{proposition}

It will be more convenient for us to prove a somewhat more involved statement, where
 the uniform
sequence is associated with the polynomial of maximal degree:
\begin{proposition}\label{P:polies'}
Let $a_1\colon \N\to \C$ be a bounded sequence and
$a_{2,N},\ldots, a_{\ell,N}\colon \N\to \C$, $N\in\N$, be a collection of uniformly bounded  sequences. Furthermore,  let $p_1,\ldots,p_\ell\in \Z[t]$ be  polynomials such that  $p_1-p_i$ is non-constant for $i=2,\ldots,\ell$, and suppose that $\text{deg}(p_1)\geq \text{deg}(p_i)$ for $i=2,\ldots,\ell$.
Then there exists $k\in \N$, depending only on  $\ell$ and on $\text{deg}(p_1)$,  such that if  $\nnorm{a_1}_{U_{k}(\N)}=0$, then
\begin{equation}\label{E:sce}
\lim_{N\to\infty}  \E_{ m\in [1, N]}\Big|\E_{n\in [1, b(N)]}
\Big(a_1(m+p_1(n))\cdot \prod_{i=2}^\ell   a_{i,N}(m+p_i(n))\Big)\Big|^2=0.
\end{equation}
\end{proposition}
\begin{proof}[Proof of Proposition~\ref{P:polies} assuming  Proposition~\ref{P:polies'}]
Let $\{p_1,\ldots,p_\ell\}$ be a family of polynomials that satisfies the assumptions of
Proposition~\ref{P:polies}. Because of the  symmetry of the statement of Proposition~\ref{P:polies}
it suffices to establish its conclusion  when $i=1$.

For $N\in \N$ we define a sequence $a_{0,N}\colon \N\to \C$ by
$$
a_{0,N}(m)\mathrel{\mathop:}=\E_{n\in [1, b(N)]}
 \prod_{i=1}^\ell   \bar{a}_i(m+p_i(n)).
$$
Then
\begin{equation}\label{E:Ntrick}
  \E_{m\in [1, N]}\Big|\E_{n\in [1, b(N)]}
\prod_{i=1}^\ell   a_i(m+p_i(n))\Big|^2=
\E_{ m\in [1, N]}\E_{ n\in [1, b(N)]}
\Big(a_{0,N}(m) \prod_{i=1}^\ell   a_i(m+p_i(n))\Big).
\end{equation}

 Let $p\in \{p_1,\ldots,p_\ell\}$ be any polynomial such that
 the polynomial $p+p_1$ has maximal degree within the family $\{p,p+p_1,\ldots,p+p_\ell\}$.
Making the change of variables $m\to m+p(n)$, and using our growth assumption $p(b(N))/N\to 0$,    we see that the difference of the averages
\begin{equation}\label{E:Ntrick1}
\E_{ n\in [1, b(N)]}\E_{ m\in [1, N]}
\Big(a_{0,N}(m) \prod_{i=1}^\ell   a_i(m+p_i(n))\Big)
\end{equation}
and the averages
\begin{equation}\label{E:Ntrick2}
\E_{n\in [1, b(N)]}\E_{m\in [1, N]}
\Big(a_{0,N}(m+p(n)) \prod_{i=1}^\ell   a_i(m+p(n)+p_i(n))\Big)
\end{equation}
converges to $0$ as $N\to \infty$.
Since by assumption the polynomials $p_1$ and $p_1-p_i$ are non-constant for $i=2,\ldots,\ell$, and   by the choice of $p$  the polynomial $p+p_1$ has maximal degree within the family $\{p,p+p_1,\ldots,p+p_\ell\}$,
the assumptions of Proposition~\ref{P:polies'} are satisfied, where the role of $p_1$ plays the polynomial $p+p_1$.
Using the Cauchy-Schwarz inequality we conclude that there exists $k\in \N$, depending only on $\ell$ and on $\text{deg}(p+p_1)$, such that if  $\nnorm{a_1}_{U_{k}(\N)}=0$, then the averages \eqref{E:Ntrick2} converge to $0$ as $N\to\infty$. As a consequence, the averages \eqref{E:Ntrick1} converge to $0$ as $N\to\infty$.
The result now follows from  \eqref{E:Ntrick}.
\end{proof}
We are going to prove Proposition~\ref{P:polies'} by repeated applications of  the following consequence of van der Corput's fundamental estimate (see, for example,  Lemma~3.1 in  \cite{KN74}):
\begin{corollary} \label{C:VDC2} Let $N\in \N$ and  $a(1),\ldots, a(N)$ be complex numbers bounded by $1$.
Then for every integer  $R$ between $1$ and $N$ we have
$$
\big|\E_{ n\in [1, N]} a(n)\big|^2\leq 4\cdot
\Big(\E_{r\in [1, R]} (1-rR^{-1}) \Re\big(\E_{ n\in [1, N]}a(n+r)\cdot \bar{a}(n)\big) + R^{-1}+RN^{-1}
\Big).
$$
\end{corollary}

\subsubsection{The linear case}
The next lemma will be used to prove the linear case of Proposition~\ref{P:polies}. Furthermore, its proof
contains
the  main technical maneuver needed to carry out the inductive step in the proof of   Proposition~\ref{P:polies}.
\begin{lemma}\label{L:U_2}
Let $a\colon \N\to \C$ be a sequence bounded by $1$.
Then
$$
\limsup_{N\to\infty} \E_{ m\in [1,N]} |\E_{n\in [1, b(N)]}   a(m+n)|^2\ll \nnorm{a}_{U_2(\N)}.
$$
\end{lemma}
\begin{proof}
Since $b(N)\to \infty$ as $N\to \infty$, by Corollary~\ref{C:VDC2} we get that for every $R\in\N$ the limit
$$
\limsup_{N\to\infty} \E_{ m\in [1, N]} |\E_{n\in [1, b(N)]}   a(m+n)|^2
$$
is bounded by $4$ times the expression
$$
  \limsup_{N\to\infty} \E_{ m\in [1,  N]} \E_{ r\in [1, R]}(1-rR^{-1})
 \Re\big( \E_{ n\in [1, b(N)]}   a(m+n+r)
\cdot \bar{a}(m+n)\big)+ R^{-1}.
$$
We interchange averages and make the change of variables $m\to m-n$.
Since
$b(N)/N\to 0$ and the sequence $(a(n))_{n\in\N}$ is bounded,  we deduce that the last expression is equal to
$$
\limsup_{N\to\infty}  \E_{ r\in [1, R]}(1-rR^{-1})
 \Re\big(
  \E_{ m\in [1, N]} a\big(m+r\big)
\cdot \bar{a}(m)\big)+ R^{-1}.
$$
Finally, letting $R\to \infty$ we get that the original limit  is bounded by
$$
\limsup_{N\to\infty}\E_{r\in [1,N]}\limsup_{N\to\infty} |\E_{  m\in [1, N]}  a(m+r)
\cdot \bar{a}(m)|\leq \nnorm{a}_{U_2(\N)}^2.
$$
Since $\nnorm{a}_{U_2(\N)}\leq 1$,  this establishes the advertised estimate.
\end{proof}

\begin{proof}[Proof of Proposition~\ref{P:polies'} for linear polynomials]
For notational convenience we let
$a_{1,N}\mathrel{\mathop:}=a_1$ for every $N\in\N$.
It suffices to  show that if  $\norm{a_{i,N}}_\infty\leq 1$ for  $i=1,\ldots,\ell$ and $N\in\N$, then
\begin{equation}\label{E:ell}
\limsup_{N\to\infty} \E_{ m\in [1, N]} \big|\E_{ n\in [1,b(N)]}  \prod_{i=1}^\ell a_{i,N}(m+k_in)\big|^2\ll_{k_1,\ldots,k_\ell}
  \nnorm{a_1}_{U_{\ell+1}(\N)}.
\end{equation}

We use induction on $\ell$, the number of sequences involved.

For  $\ell=1$ the result follows from Lemma~\ref{L:U_2}, and the estimate
\begin{equation}\label{E:trivial}
\limsup_{N\to\infty}\E_{n\in [1,N]}|a(kn)|\leq k \cdot \limsup_{N\to\infty}\E_{n\in [1,N]}|a(n)|.
\end{equation}

To carry out the  inductive step, let $\ell\geq 2$, and suppose that the statement holds for $\ell-1$ sequences.
Following the argument used in the proof of Lemma~\ref{L:U_2},  using the induction hypothesis, and the estimate \eqref{E:trivial},
we get that
 the left hand side in \eqref{E:ell} is bounded by a constant, that depends on $k_1,\ldots, k_\ell$, multiple of
$$
\limsup_{N\to\infty}\E_{r\in [1,N]} \nnorm{S_ra_1\cdot \bar{a}_1}_{U_\ell(\N)}\leq  \nnorm{a_1}^2_{U_{\ell+1}(\N)}
$$
where the last estimate follows from \eqref{E:inductive} and H\"{o}lder's inequality.
Since $\norm{a_1}_\infty\leq 1$, we have
$\nnorm{a_1}_{U_{\ell+1}(\N)}\leq 1$.
This completes the proof.
\end{proof}
\subsubsection{The general case}
We first explain an induction scheme, often called
PET induction (Polynomial Exhaustion Technique), on types of
families of polynomials that was introduced by Bergelson in \cite{Be87a}.

We define the \emph{degree} of  a family $\mathcal{P}$ of non-constant polynomials
to be the maximum of the degrees of the polynomials in the family.
Let $\mathcal{P}_i$ be the subfamily of polynomials of degree $i$ in
$\mathcal{P}$. We let $w_i$ denote the number of distinct leading
coefficients that appear in the family $\mathcal{P}_i$. The vector
$(d,w_d,\ldots,w_1)$ is called the \emph{type} of the family of
polynomials $\mathcal{P}$.
We order the set of all possible types lexicographically, meaning,
$(d,w_d,\ldots,w_1)>(d',w_{d'}',\ldots,w_1')$ if and only if in the
first instance where the two vectors disagree the coordinate of the
first vector is greater than the coordinate of the second vector.
One easily verifies that every decreasing sequence of types is eventually constant,
thus, if some operation reduces the type, then after a finite number of repetitions it
is going to terminate.

Next, we define such an operation:
Let  $\mathcal{P}=(p_1,\ldots,p_\ell)$ be an ordered family of polynomials,
  $p\in \mathcal{P}$, and $r\in\N$.
The family $(p,r)\vdc(\mathcal{P})$ consists of all non-constant polynomials
of the form $p_i-p$, $S_rp_i-p$, $i=1,\ldots, \ell$,
where $S_rp$ is defined by $(S_rp)(n)\mathrel{\mathop:}=p(n+r)$. We order them
so that the polynomial $S_rp_1-p$ appears first.

We call an ordered family of polynomials  $(p_1,\ldots,p_\ell)$ \emph{nice} if
$\text{deg}(p_1)\geq \text{deg}(p_i)$ and
$p_1-p_i$ is non-constant for
$i=2,\ldots, \ell$.
\begin{lemma}\label{L:IndVDC}
  Let $\mathcal{P}=(p_1,\ldots,p_\ell)$ be a  nice ordered family of polynomials, and suppose that $\text{deg}(p_1)\geq 2$.
  Then there exists a polynomial $p\in \mathcal{P}$, such that for every large enough $r\in\N$, the family $(p,r)\vdc(\mathcal{P})$ is nice and has strictly smaller type than that of $\mathcal{P}$.
\end{lemma}
\begin{proof}
If all the polynomials
have the same degree and leading coefficient, then we take $p=p_1$.
If all the polynomials have the same degree and at least one has different leading coefficient than $p_1$, then  we take any such polynomial as $p$.
Otherwise, there exists a non-constant polynomial in $\mathcal{P}$ with degree strictly smaller than the degree of $p_1$. We  take $p$ to
be any such polynomial that has minimal degree.
In all cases, it is easy to check the advertised property.
\end{proof}

\begin{proof}[Proof of Proposition~\ref{P:polies'}]
It suffices to show that the $k$ given in the statement of Proposition~\ref{P:polies'}
depends only on the number $\ell$ and the type $W$ of the family of polynomials   involved.
This is the case because if we fix the degree and the cardinality of a family of polynomials, then there
 are a finite number of possibilities for its  type.

We are going to use induction on the type of the family of polynomials involved.
As our base case we take the case where all the polynomials are linear; then the result was proved in  the previous
subsection with $k=\ell+1$.

Let now $\mathcal{P}$ be a nice  ordered family of $\ell$ polynomials  with $\text{deg}(p_1)\geq 2$ and type $W$, and suppose that the statement
holds for all nice ordered families of $\ell'$ polynomials with type $W'$ strictly smaller than $W$ for some  $k=k(W',\ell')\in \N$.

 Let $p\in \mathcal{P}$ be chosen as in Lemma~\ref{L:IndVDC}.
Using  Corollary~\ref{C:VDC2},  making the change of variables $m\to m-p(n)$, and using that $p(b(N))/N\to 0$, exactly as in the proof of Lemma~\ref{L:U_2},  we get that the limsup as $N\to\infty$ of the averages in \eqref{E:sce}
 is bounded by a constant multiple of
 $$
\limsup_{N\to\infty}\E_{r\in [1,N]}
\limsup_{N\to\infty}\E_{m\in [1, N]}\big|\E_{ n\in [1, b(N)]}
\prod_{i=1}^\ell \bar{a}_{i,N}\big(m+p_i(n+r)-p(n)\big)\cdot a_{i,N}\big(m+p_i(n)-p(n)\big)\big|,
$$
where again for notational convenience we have defined
$a_{1,N}\mathrel{\mathop:}=a_1$ for $N\in\N$.
By Lemma~\ref{L:IndVDC},
 for suitably large  $r\in\N$, the family $(p,r)\vdc(\mathcal{P})$ is nice, has type strictly smaller than $W$, and consists of at most $2\ell$ polynomials. Let
 $$
 k(W,\ell)=\max_{W'<W,\ell'\leq 2\ell}k(W',\ell'),
 $$
 where the maximum is taken over all $\ell'$ with $ \ell'\leq 2\ell$ and possible types $W'$ with $W'<W$ of families consisting of at most $2\ell$ polynomials (there is a finite number of such possible types).
Using the induction hypothesis and the Cauchy-Schwarz inequality, we get that if $\nnorm{a_1}_{U_{k(W,\ell)}(\N)}=0$, then  for every large enough $r\in \N$ we have
$$
\limsup_{N\to\infty}\E_{ m\in [1, N]}\Big|\E_{ n\in [1, b(N)]}
\Big(\prod_{i=1}^\ell \bar{a}_{i,N}\big(m+p_i(n+r)-p(n)\big)\cdot a_{i,N}\big(m+p_i(n)-p(n)\big)\Big)\Big|=0.
$$
This completes the induction and the proof.
\end{proof}
\subsection{Proof of the main results for polynomial averages}
We are now one short step from proving  Theorems~\ref{T:PoliesConv1}  and \ref{T:PoliesChar}.

\begin{proof}[Proof of Theorem~\ref{T:PoliesChar}]
Let $k\in \N$ be such that the conclusion of Proposition~\ref{P:polies} holds.
Without loss of generality we can assume that $\nnorm{f_1}_{k,\mu,T_1}=0$. Then  for $\mu$ almost every $x\in X$ we have
$\nnorm{f_1}_{k,\mu_x,T_1}=0$. Using Proposition~\ref{P:RelToErgodic} we deduce that
for $\mu$ almost every
$x\in X$  we have
$\nnorm{f_1(T_1^nx)}_{U_k(\N)}=0$.
The result now follows by applying  Proposition~\ref{P:polies} to the sequences $a_i\colon \N\to \C$ defined
by $a_i(n)\mathrel{\mathop:}=f_i(T_i^nx)$, $i=1,\ldots,\ell$.
\end{proof}
\begin{proof}[Proof of Theorem~\ref{T:PoliesConv1}]
We assume as we may that $\norm{f_i}_{L^\infty(\mu)}\leq 1$ for $i=1,\ldots,\ell$.
Let $k\in \N$ be the integer given by Theorem~\ref{T:PoliesChar}. Let $\varepsilon>0$.
For $i=1\ldots,\ell$, we use Proposition~\ref{P:ApprNil} to get the decomposition $f_i=f^s_{i,\varepsilon}+f^u_{i,\varepsilon}+f^e_{i,\varepsilon}$, where
$\nnorm{f^u_{i,\varepsilon}}_{k}=0$,   $\norm{f^e_{i,\varepsilon}}_{L^1(\mu)}\leq \varepsilon$,
all functions are bounded by $2$, and
for $\mu$ almost every $x\in X$, the sequence $(f^s_{i,\varepsilon} (T^nx))_{n\in\N}$
is a $(k-1)$-step nilsequence.
Let  $$
A_N(f_i)(x)\mathrel{\mathop:}=\E_{ m\in [1,  N], n\in [1, b(N)]} f_1(T_1^{m+p_1(n)}x)\cdot \ldots\cdot f_\ell(T_\ell^{m+p_\ell(n)}x).
$$
Theorem~\ref{T:ConvNil} implies  that
 the averages $A_N(f^s_{i,\varepsilon})(x)$ convergence pointwise.
 Hence, it suffices  to show  that
when computing the average
    $A_N(f_i)(x)$    the contribution of the functions $f^u_{i,\varepsilon}$ and $f^e_{i,\varepsilon}$
    becomes negligible  as $N\to\infty$ and $\varepsilon$ is taken suitably small.
    Theorem~\ref{T:PoliesChar} implies that the contribution of   the functions  $f^u_{i,\varepsilon}$
is negligible, independently of the choice of $\varepsilon$.
To handle the  contribution of the functions $f^e_{i,k}$ we argue as  in the proof of the corresponding convergence result for
the cubic averages in Section~\ref{SS:2.2}.   Let us just  explain the only point where our argument deviates slightly from the aforementioned  argument. We expand $
A_N(f^s_{i,\varepsilon}+f^e_{i,\varepsilon})$ and write
$
A_N(f^s_{i,\varepsilon}+f^e_{i,\varepsilon})-A_N(f^s_{i,\varepsilon})
$
as a sum  of $2^\ell -1$  averages. We deal with each such average separately, and bound all the functions
by their sup norm
 except  one (chosen arbitrarily)  that
is equal to $f^e_{i,\varepsilon}$ for some $i\in \{1,\ldots,\ell\}$.
Upon doing this, we get the bound
 $$
| A_N(f^s_{i,\varepsilon}+f^e_{i,\varepsilon})(x)-A_N(f^s_{i,\varepsilon})(x)|
\ll_\ell \max_{i=1,\ldots,\ell}\E_{m\in [1,N], n\in [1,b(N)]} |f_{i,\varepsilon}^e|({T_i^{m+p_i(n)}}x)
 $$
That the right hand side  becomes negligible as $N\to\infty$, and $\varepsilon$ is chosen suitably small,
follows (as in the proof given in  Section~\ref{SS:2.2}) upon noticing  that for every
system $(X,\X,\mu, T)$, function $f\in L^\infty(\mu)$, and polynomial $p\in \{p_1,\ldots,p_\ell\}$,
one has for $\mu$ almsot every $x\in X$ that
$$
\lim_{N\to\infty}\E_{ m\in [1, N], n\in [1, b(N)]}
|f|(T^{m+p(n)}x)=
\int |f| \ d\mu_x
$$
where $\mu=\int \mu_x \ d\mu(x)$ is the ergodic decomposition of the measure $\mu$ with respect to $T$.
To get this identity it suffices to make the change of variables $m\to m-p(n)$, use  that
$p(b(N))/N\to 0$ and the ergodic theorem. This completes the proof.
\end{proof}

\end{document}